\documentclass[11pt]{article}

\usepackage[margin=1.1in]{geometry}
\usepackage{amsmath,amssymb,amsthm,mathtools}
\usepackage{thm-restate}
\usepackage{bm}
\usepackage{enumitem}
\usepackage{microtype}
\usepackage[numbers]{natbib}
\usepackage{xcolor}
\usepackage[colorlinks=true,citecolor=blue!60!black,linkcolor=blue!60!black,urlcolor=blue!60!black]{hyperref}

\newtheorem{theorem}{Theorem}
\newtheorem{proposition}[theorem]{Proposition}
\newtheorem{lemma}[theorem]{Lemma}

\newtheorem{definition}[theorem]{Definition}

\newtheorem{question}{Question}
\newtheorem{remark}[theorem]{Remark}

\newcommand{\R}{\mathbb{R}}
\newcommand{\Q}{\mathbb{Q}}
\newcommand{\supp}{\operatorname{supp}}
\newcommand{\Crit}{\operatorname{Crit}}
\newcommand{\conv}{\operatorname{conv}}
\newcommand{\trdeg}{\operatorname{trdeg}}
\newcommand{\argmax}{\operatorname*{arg\,max}}

\title{On Finite Gaussian Mixtures: Finiteness of the Number of Modes and an Application to NPMLE}
\author{Haiyang Wang\thanks{Department of Applied and Computational Mathematics, Yale University, New Haven, CT 06511,
USA (e-mail: haiyang.wang1024@gmail.com)}}
\date{\today}

\begin{document}

\maketitle

\begin{abstract}
We prove that every isotropic Gaussian mixture with finitely many components has finitely many modes.
In one dimension, classical theory of Chebyshev systems  gives the sharp bound of at most \(n\) modes for an \(n\)-component mixture. In several dimensions, however, it has remained open whether every such mixture has finitely many modes.
Our main result is the stronger statement that the entire critical set has finite cardinality,
which is proven by combining real analytic curve selection theorem and Ax's functional-transcendence theorem. 

As an application, we show that, for Gaussian location mixtures, every nonparametric maximum likelihood estimator (NPMLE) based on a finite dataset is finitely supported.
More specifically, all NPMLEs share the same finite set of allowable atom locations.

\end{abstract}

\section{Introduction}
\label{sec:introduction}


For \(d\geq1\), given centers \(x_1,\ldots,x_n\in\R^d\) and positive weights \(w_1,\ldots,w_n\),  the \(n\)-component Gaussian mixture is  defined by
\begin{equation}
\label{eq:intro-mixture}
F(\theta)=\sum_{i=1}^{n}w_i\varphi_d(\theta-x_i),
\end{equation}
where \(\varphi_d(x)=(2\pi)^{-d/2}\exp\left(-\frac12\lVert x\rVert^2\right)\) is the density of \(\mathcal{N}(0,I_d)\). 

More generally, an \(n\)-component heteroscedastic Gaussian mixture allows a different covariance matrix \(\Sigma_i\) at each center \(x_i\), giving \(F(\theta)=\sum_{i=1}^{n}w_i\det(\Sigma_i)^{-1/2}\varphi_d(\Sigma_i^{-1/2}(\theta-x_i))\). A homoscedastic Gaussian mixture with covariance \(\Sigma\) is the special case in which \(\Sigma_i=\Sigma\) for every \(i\in[n]\). Under the linear change of variables \(x\mapsto\Sigma^{-1/2}x\), such a mixture is equivalent to the isotropic model in \eqref{eq:intro-mixture}. This paper focuses on the isotropic model, and our results readily generalize to homoscedastic Gaussian mixtures and may also extend to heteroscedastic Gaussian mixtures.

The modes of a finite Gaussian mixture \(F\) are key objects in the analysis of several applications and algorithms,
including nonparametric maximum likelihood estimation (NPMLE) and clustering with Gaussian kernels \cite{ghassabeh2015,kieferWolfowitz1956,lindsayRoeder1993,liRayLindsay2007,wallace2013,yanWangRigollet2024,sahaGuntuboyina2020}.
Central to these analyses is mode counting, which helps characterize the algorithmic difficulty and structural complexity of clustering and NPMLE problems.
\begin{question}
Does the finite Gaussian mixture \(F\) have a finite number of modes? If so, how many modes can it have?
\end{question}
We answer the first, qualitative question affirmatively by proving the stronger result that every finite Gaussian mixture \(F\) of the form \eqref{eq:intro-mixture} has only finitely many critical points.

\begin{restatable}{theorem}{maintheorem}
\label{thm:critical-finite}
For every \(d,n\geq1\), \(x_1,\ldots,x_n\in\R^d\), and \(w_1,\ldots,w_n>0\), let \(F(\theta)\) be the Gaussian mixture defined in \eqref{eq:intro-mixture}. Its set of critical points
\[
\Crit(F)\coloneqq\{\theta\in\R^d:\nabla F(\theta)=0\}
\]
has finite cardinality. Consequently, \(F\) has a finite number of modes.
\end{restatable}

\subsection{Related Work}

The scalar case \(d=1\) is well understood.
The classical theory of Chebyshev systems implies that an \(n\)-component Gaussian mixture on \(\R\) has at most \(n\) modes and \(2n-1\) critical points \cite{borweinChebyshevDescartesSystems1995, carreiraPerpinanWilliams2003,polyanskiyWu2020}.
The situation changes drastically in \(d>1\), where the conjecture that the critical set is finite has remained open \cite{wallace2013}.

For \(n=2\) or \(3\), sharp quantitative results have been established. For two-component heteroscedastic Gaussian mixtures in \(\R^d\), Ray and Ren proved the sharp upper bound of \(d+1\) on the number of modes \cite{rayRen2012}.
Okuno and Kabata recently proved that every three-component homoscedastic Gaussian mixture has at most \(15\) critical points and at most \(8\) modes \cite{okunoKabata2026}.

For general \(d\) and \(n\), examples show that the number of modes or critical points can be much greater than \(n\).
Carreira-Perpi{\~n}{\'a}n and Williams exhibited homoscedastic mixtures with more modes than components \cite{carreiraPerpinanWilliams2003}.
Edelsbrunner, Fasy, and Rote analyzed the regular-simplex construction, found exponentially many critical points, and constructed finite isotropic mixtures with a superlinear number of modes \cite{edelsbrunnerFasyRote2013}.
In contrast to the case \(d=1\), the multivariate setting \(d>1\) permits much faster growth in the numbers of modes and critical points. In particular, even for \(d=2\), it was not known whether a finite Gaussian mixture \(F\) necessarily has finitely many modes or critical points.

Under the assumption that the critical set of \(F\) is finite, several quantitative results bound the numbers of modes and critical points.
Am\'endola, Engstr\"om, and Haase obtained a lower bound and the first upper bound on the maximum number of modes for general \(n\) and \(d\), covering both heteroscedastic and homoscedastic mixtures \cite{amendola2020}.
Nguyen subsequently improved these lower and upper bounds \cite{nguyen2026}. Theorem~\ref{thm:critical-finite} supplies the finiteness assumption used in these quantitative results, so the resulting mode-counting bounds hold unconditionally.

\subsection{Application}

For clustering, the Gaussian mean-shift algorithm seeks modes of a Gaussian kernel density estimate. Its convergence was established under the assumption that the stationary points of the estimate, which are critical points of a finite Gaussian mixture, are isolated \cite{ghassabeh2015}. Since a Gaussian kernel density estimate based on finitely many observations becomes a finite isotropic Gaussian mixture after rescaling by the bandwidth, Theorem~\ref{thm:critical-finite} shows that this assumption always holds and thereby guarantees convergence of the Gaussian mean-shift algorithm.

The Gaussian-mixture NPMLE, discussed in more detail in Section~\ref{sec:npmle}, is an effective and parameter-free estimator with many nearly-optimal statistical properties \cite{chen2026sharpregrethellingerboundsgaussian,sahaGuntuboyina2020,yanWangRigollet2024}.
However, NPMLEs can be nonunique \cite{soloff2025}, and previous work left open whether an NPMLE could have a continuum of support points, a possibility that could cause algorithmic difficulties. By applying Theorem~\ref{thm:critical-finite}, we prove that every NPMLE is finitely supported.
\begin{restatable}{theorem}{maintheoremnpmle}
\label{thm:npmle-finite}
Let \(d,n\geq1\) and \(x_1,\ldots,x_n\in\R^d\).
There exists a finite Gaussian mixture \(F\) such that every solution \(\widehat\pi\) to the NPMLE problem \eqref{eq:npmle_definition} for the dataset \(\{x_1,\ldots,x_n\}\) is supported on \(\Crit(F)\), which has finite cardinality by Theorem~\ref{thm:critical-finite}. Therefore, every \(\widehat\pi\) is a discrete distribution with finite support.
\end{restatable}

The rest of the paper is organized as follows. In Section~\ref{sec:critical}, we recall the necessary background in analytic geometry and Ax--Schanuel functional-transcendence theory and then prove Theorem~\ref{thm:critical-finite}. In Section~\ref{sec:npmle}, we review the NPMLE for Gaussian location-mixture models and several known propositions before applying Theorem~\ref{thm:critical-finite} to prove Theorem~\ref{thm:npmle-finite}. In Section~\ref{sec:discussion}, we discuss limitations, possible extensions, and open problems.

\section{Proof of Theorem~\ref{thm:critical-finite}}
\label{sec:critical}

In this section, we first restate Theorem~\ref{thm:critical-finite}, sketch its proof, recall the necessary mathematical background, and then give the full proof.

\maintheorem*

\paragraph{Outline of the proof.}
We prove by contradiction, assuming that \(|\Crit(F)|=\infty\).
The critical points of \(F\) lie in the compact convex hull of
\(\{x_1,\ldots,x_n\}\).
Thus \(|\Crit(F)|=\infty\) implies that \(\Crit(F)\) has an accumulation point \(\theta_*\in\Crit(F)\).
The real-analytic curve selection theorem then produces an analytic curve \(\eta:[0,1]\to\Crit(F)\) with \(\eta(0)=\theta_*\).
We restrict the relevant functions and critical-point equations to the curve \(\eta\). More specifically, we study the transcendence degree over \(\R\) of the function field
\[\R\big(x_1\cdot\eta(t),\ldots,x_n\cdot\eta(t),\exp(x_1\cdot\eta(t)),\ldots,\exp(x_n\cdot\eta(t))\big),\]
which lies in the fraction field of the ring of real-analytic functions on \((0,1)\).
The critical equation \(\nabla F(\theta)=0\) holds along \(\eta\) and imposes an upper bound on the transcendence degree of this field. On the other hand, Ax's functional-transcendence theorem gives an incompatible lower bound. This contradiction proves that \(|\Crit(F)|\) is finite.

For the remainder of this section, we review the necessary background in analytic geometry, field theory, differential algebra, and Ax--Schanuel functional-transcendence theory before proving the main theorem.

\subsection{Preliminaries on real-analytic geometry}

Analytic geometry studies spaces that can be described locally as zero sets of analytic functions, in both the real-analytic and complex-analytic settings.
General introductions include \cite{krantzParks2002} for real-analytic functions and \cite{narasimhan1966,chirka1989,lojasiewicz1965} for real-analytic and complex-analytic sets.
In this paper, for simplicity, we restrict attention to real-analytic subsets of Euclidean space.
A more general treatment of analytic geometry, including semianalytic and subanalytic sets, can be found in \cite{lojasiewicz1965,bierstoneMilman1988}.

\begin{definition}[Real-analytic function on \(\R^d\)]
On an open set \(U\subseteq\R^d\),
a function \(f:U\to\R\) is real analytic if \(f\in C^\infty(U)\) and, for every \(x_0\in U\), there exists an open neighborhood \(V\ni x_0\) such that the Taylor series of \(f\) at \(x_0\) converges to \(f\) on \(V\). The set of real-analytic functions on \(U\), denoted by \(\mathcal O(U)\), is a ring.

\end{definition}

\begin{definition}[Real-analytic subset of \(\R^d\)]
\(A\subseteq\R^d\) is a real-analytic subset of \(\R^d\) if, for every \(x_0\in\R^d\), there is an open neighborhood \(U\) of \(x_0\) and finitely many real-analytic functions \(f_1,\ldots,f_r:U\to\R\) such that
\[
A\cap U=\{x\in U:f_1(x)=\cdots=f_r(x)=0\}.
\]

In other words, locally at each point \(x_0\in\R^d\), \(A\) is the common zero set of finitely many real-analytic functions.
\end{definition}


\begin{theorem}[Real-analytic curve selection]
\label{thm:curve-selection}
Let \(A\) be a real-analytic subset of \(\R^d\),
and let \(p\in A\) be a non-isolated point;
that is, there exists a sequence \((p_m)_{m=1}^\infty\) in \(A\setminus\{p\}\)
such that \(\lim_{m\to\infty}p_m=p\).
There exists a continuous injective map \(\eta:[0,1]\to\R^d\) such that
\[
\eta(0)=p,
\qquad
\eta((0,1])\subseteq A\setminus\{p\}.
\]
Furthermore, each coordinate function \(\eta_k\), for \(k=1,\ldots,d\), is analytic on \((0,1)\).
\end{theorem}

Theorem~\ref{thm:curve-selection} appears in \cite[Section~19, Proposition~2]{lojasiewicz1965}, where it is stated in the more general setting in which \(A\) is a real semianalytic manifold.
For simplicity, we restrict the original statement in \cite{lojasiewicz1965} to the case in which \(A\) is a real-analytic subset of Euclidean space.

\subsection{Preliminaries on Ax--Schanuel functional-transcendence theory}

Before stating Ax's functional-transcendence theorem, we introduce a few algebraic concepts.
Algebraic independence and transcendence degree measure polynomial relations among elements of a field extension; a standard reference is \cite{lang2002}.
We also recall the differential algebra needed below; see \cite{kolchin1973}.
The functional-transcendence theorem used here is due to Ax \cite{ax1971}.

Let \(K\subseteq L\) be fields. We recall the concepts of algebraic independence and transcendence degree.

\begin{definition}[Algebraic independence]
Elements \(\alpha_1,\ldots,\alpha_s\in L\) are {algebraically independent over \(K\)} if, for every polynomial \(P\in K[X_1,\ldots,X_s]\),
\[
P(\alpha_1,\ldots,\alpha_s)=0 \implies P=0.
\]
\end{definition}

\begin{definition}[Transcendence degree]
A maximal subset of \(L\) that is algebraically independent over \(K\) is a transcendence basis of the field extension \(L/K\).
Its cardinality is the {transcendence degree}, denoted by \(\trdeg_K L\).
\end{definition}
\begin{remark}
    By \cite[Chapter~VIII, Theorem~1.1, p.~356]{lang2002}, distinct transcendence bases of \(L/K\) have the same cardinality.
    Hence \(\trdeg_K L\) is independent of the choice of transcendence basis.

    For readers unfamiliar with transcendence degree, a few elementary examples may be helpful:
    \(\trdeg_{\R}\mathbb C=0\) because \(\mathbb{C}=\R(i)\) and \(i\) satisfies the real-coefficient polynomial equation \(x^2+1=0\); \(\trdeg_{\Q}\Q(\pi)=1\) because \(\pi\) is transcendental; and \(\trdeg_{\Q}\R=\infty\) because \(\Q\) is countable whereas \(\R\) is uncountable.
    Determining the exact value of \(\trdeg_{\Q}\Q(\pi,e)\in\{1,2\}\) remains open; it would equal \(2\) under Schanuel's conjecture.
\end{remark}

Ax--Schanuel functional-transcendence theory was originally developed for function fields. Such fields carry the additional structure of differentiation, which can be characterized abstractly as follows \cite{kolchin1973}.
\begin{definition}[Derivation on a field]
A {derivation} \(\partial\) on a field \(L\) is a map \(\partial:L\to L\) that is additive and satisfies the Leibniz rule
\[
\partial(x+y)=\partial(x)+\partial(y),\qquad \partial(xy)=x\partial(y)+y\partial(x),\qquad \forall x,y\in L.
\]

\end{definition}

\begin{definition}[Constant field of a differential field]
    A differential field is a field \(L\) equipped with \(\Delta\), a family of derivations on \(L\). Its constant field is defined to be
\[
C=\{x\in L:\partial x=0\text{ for every }\partial\in\Delta\}.
\]
\end{definition}
\begin{remark}
    If \((L,\Delta)\) is a differential field of characteristic zero, then \(\Q\subseteq C\); see also \cite[Chapter~I, Section~1]{kolchin1973}.
    Indeed, for every \(\partial\in\Delta\), the Leibniz rule gives
    \[
    \partial(1)=\partial(1\cdot1)=2\partial(1),
    \]
    so \(\partial(1)=0\). Additivity then gives \(\partial(m)=0\) for every \(m\in\mathbb Z\). If \(n\in\mathbb Z\setminus\{0\}\), then
    \[
    0=\partial(1)=\partial(n n^{-1})=n\partial(n^{-1})+n^{-1}\partial(n)=n\partial(n^{-1}),
    \]
    and hence \(\partial(n^{-1})=0\). Thus every element of the prime field \(\Q\subseteq L\) is annihilated by every \(\partial\in\Delta\), proving \(\Q\subseteq C\).
\end{remark}

\begin{definition}[Rational independence]
Let \(\Q\subseteq K\subseteq L\) be fields.
Elements \(\alpha_1,\ldots,\alpha_s\in L\) are {rationally independent modulo \(K\)} if, for every
\(q_1,\ldots,q_s\in\Q\),
\[
\sum_{j=1}^{s}q_j\alpha_j \in K \implies q_1=q_2=\cdots=q_s=0.
\]
\end{definition}

With this differential-algebraic background in place, we can now state Ax's functional-transcendence theorem.

\begin{theorem}[Ax's functional-transcendence theorem]
\label{thm:ax-general}
Let \((L,\Delta)\) be a differential field of characteristic zero with constant field \(C\).
For \(s\geq1\), suppose that \(\alpha_1,\ldots,\alpha_s\in L\) are rationally independent modulo \(C\) and that \(\beta_1,\ldots,\beta_s\in L^\times\) satisfy
\[
\frac{\partial\beta_j}{\beta_j}=\partial\alpha_j,
\qquad \forall \partial\in\Delta,\quad \forall j=1,\ldots,s.
\]
Then
\[
\trdeg_C C(\alpha_1,\ldots,\alpha_s,\beta_1,\ldots,\beta_s)
\geq
s+\operatorname{rank}
\bigl(\partial\alpha_j\bigr)_{\partial\in\Delta, 1\leq j\leq s},
\]
where the last term is the rank of the derivative matrix \(\bigl(\partial\alpha_j\bigr)_{{\partial\in\Delta, 1\leq j\leq s}}\in L^{|\Delta|\times s}\).
\end{theorem}

This theorem first appeared in \cite[Theorem~3, p.~253]{ax1971}. See also \cite{pong2016} for a modern exposition.

Finally, we introduce the following analytic function field on an open interval \(I\subset\R\),
which we will use to apply Ax's functional-transcendence theorem.
\begin{definition}[Analytic function field on an open interval]
Let \(I\subseteq\R\) be an open interval and let \(\mathcal O(I)\) be the ring of real-analytic functions on \(I\).
The {analytic function field on \(I\)} is
\[
\mathcal M(I)\coloneqq\operatorname{Frac}\mathcal O(I).
\]
The fact that this fraction field is well-defined is established in part \textup{(i)} of the following lemma.
\end{definition}

\begin{lemma}[Basic properties of the analytic function field]
\label{lem:analytic-field}
For an open interval \(I\subset \mathbb{R}\), the following properties hold:
\begin{enumerate}[label=\textup{(\roman*)}]
    \item \(\mathcal O(I)\) is an integral domain, so its fraction field \(\mathcal{M}(I)\) is well-defined.
    \item The derivation \(\partial_0:\mathcal O(I)\to\mathcal O(I)\) defined by \(\partial_0(f)=f'\) extends uniquely to a derivation \(\partial\) on \(\mathcal M(I)\) by the quotient rule
    \[
    \partial\!\left(\frac{f}{g}\right)=\frac{f'g-fg'}{g^2}.
    \]
    \item The constant field of \((\mathcal M(I),\partial)\) is \(\R\).
\end{enumerate}
\end{lemma}

\begin{proof}
For \textup{(i)}, suppose that \(fg=0\) for \(f,g\in\mathcal O(I)\) and that \(f\) is not identically zero.
Continuity gives a nonempty open subinterval on which \(f\neq0\), so \(g\) vanishes there.
The identity theorem for real-analytic functions \cite[Corollary~1.2.6, p.~14]{krantzParks2002} implies that \(g\) vanishes on all of \(I\).

For \textup{(ii)}, termwise differentiation of the local power-series expansion of \(f\) shows that \(f'\) is again real analytic; see \cite[Chapter~1, Section~1.1]{krantzParks2002}.
Thus \(\partial_0\) maps \(\mathcal O(I)\) into itself, and the usual sum and product rules show that it is additive and satisfies the Leibniz rule. The general extension of a derivation to a ring of quotients is given in \cite[Chapter~I, Section~3]{kolchin1973}.

For \textup{(iii)}, let \(h=f/g\in\mathcal M(I)\) satisfy \(\partial h=0\).
Since \(g\) is not identically zero, there is a nonempty open subinterval \(J\subseteq I\) on which \(g\neq0\).
On \(J\), \(\left(f/g\right)'=0\), so \(f/g=c\) on \(J\) for some \(c\in\R\).
Thus \(f-cg\) vanishes on \(J\), and the identity theorem \cite[Corollary~1.2.6, p.~14]{krantzParks2002} implies that it vanishes on all of \(I\). Hence \(h=c\) in \(\mathcal M(I)\). Conversely, every element of \(\R\) is annihilated by \(\partial\), so the constant field is exactly \(\R\).
\end{proof}

\subsection{Proof of Theorem~\ref{thm:critical-finite}}

\begin{proof}[Proof of Theorem~\ref{thm:critical-finite}]
Suppose for contradiction that \(|\Crit(F)|=\infty\).

\paragraph{Critical equation.} Rewrite the mixture in \eqref{eq:intro-mixture} as
\[
F(\theta)=\varphi_d(\theta) M(\theta),
\]
where \(M(\theta)=\sum_{i=1}^{n}b_i\exp(x_i\cdot \theta)\) and
\(b_i=w_i\exp\left(-\frac12\lVert x_i\rVert^2\right)>0\). Taking the gradient yields
\[
\nabla F(\theta)
=\varphi_d(\theta)
\left(
\sum_{i=1}^{n}b_i \exp(x_i\cdot \theta) x_i
-M(\theta)\theta
\right).
\]
Because \(M(\theta)\) and \(\varphi_d(\theta)\) are positive, the critical equation \(\nabla F(\theta)=0\) is equivalent to
\begin{equation}
\label{eq:critical-fixed-point}
\theta
=
\sum_{i=1}^{n}\frac{b_i \exp(x_i\cdot \theta)}
{M(\theta)}x_i.
\end{equation}
This is a system of \(d\) real-analytic equations in \(\theta\in\R^d\). Hence \(\Crit(F)\) is their common zero set and is therefore a real-analytic subset of \(\R^d\).
Moreover, the coefficients on the right-hand side of \eqref{eq:critical-fixed-point} are positive and sum to one.
Hence
\[
\Crit(F)\subseteq\conv\{x_1,\ldots,x_n\}.
\]
The critical set is closed and lies in this compact convex hull, so it is compact.

\paragraph{Curve selection.} Since \(\Crit(F)\) is infinite and compact, it has an accumulation point \(\theta_*\in\Crit(F)\).
By Theorem~\ref{thm:curve-selection}, there exists a continuous injective curve \(\eta:[0,1]\to\Crit(F)\) such that \(\eta(0)=\theta_*\) and each coordinate function \(\eta_k\) is real analytic on \(I=(0,1)\).

\paragraph{Restriction to the curve.}
Write \(\eta=(\eta_1,\ldots,\eta_d)\) and \(x_i=(x_{i1},\ldots,x_{id})\).
Define the exponent in \eqref{eq:critical-fixed-point} along \(\eta\) by
\[
g_i(t)=x_i\cdot \eta(t)
=\sum_{k=1}^{d}x_{ik}\eta_k(t),
\qquad i=1,\ldots,n.
\]
Since each \(\eta_k\) is real analytic on \(I=(0,1)\), every \(g_i\), as a linear combination of the \(\eta_k\), is also real analytic on \(I\).

Let \(\mathcal O(I)\) be the ring of real-analytic functions on \(I=(0,1)\), and let \(\mathcal M(I)=\operatorname{Frac}\mathcal O(I)\) be its fraction field, equipped with the derivation \(\partial=\frac{d}{dt}\). By Lemma~\ref{lem:analytic-field}, the constant field of \(\mathcal M(I)\) is \(\R\).
Moreover, \(\eta_k(t),g_i(t),\exp(g_i(t))\in\mathcal O(I)\subseteq\mathcal M(I)\). 

This is the restriction-to-the-curve step: after composition with \(\eta\), the functions \(x_i\cdot\theta\) and \(\exp(x_i\cdot\theta)\) become functions on \(I\), and the critical equation \eqref{eq:critical-fixed-point} becomes an identity on \(I\).

Choose a maximal subset \(\{v_1,\ldots,v_s\}\) of \(\{g_1,\ldots,g_n\}\) that is rationally independent modulo the constant field \(\R\) of \(\mathcal M(I)\). The case \(s=0\) is impossible: it would imply that every \(g_i\) lies in \(\R\), so evaluating \eqref{eq:critical-fixed-point} along \(\eta\) would force \(\eta\) to be constant, contrary to its injectivity. Thus \(s\geq1\).
We are interested in the transcendence degree of the following field over \(\R\):
\[
L=\R(v_1,\ldots,v_s,\exp(v_1),\ldots,\exp(v_s)).
\]
We will use the critical equations to prove the upper bound \(\trdeg_{\R}L\leq s\) and Ax's theorem \cite{ax1971} to prove the incompatible lower bound \(\trdeg_{\R}L\geq s+1\).

\paragraph{Upper bound.}
Restricting the critical equation \eqref{eq:critical-fixed-point} along the curve \(\eta\) gives
\[
\eta(t)=\frac{\sum_{i=1}^{n}b_i\exp(x_i\cdot\eta(t))x_i}{\sum_{i=1}^{n}b_i\exp(x_i\cdot\eta(t))},
\]
which implies that
\[
\eta_k\in\R(\exp(g_1),\ldots,\exp(g_n)),\quad k=1,\ldots,d.
\]
Furthermore, since \(g_i\) is a linear combination of \(\{\eta_k\}_{k=1}^d\),
\[
g_i\in\R(\exp(g_1),\ldots,\exp(g_n)),\quad i=1,\ldots,n.
\]
It follows that
\begin{equation}
    L \subseteq \R(g_1,\ldots,g_n,\exp(g_1),\ldots,\exp(g_n)) = \R(\exp(g_1),\ldots,\exp(g_n)).
    \label{eq:subfield}
\end{equation}

On the other hand, by the maximality of \(v_1,\ldots,v_s\), for each \(i\), there exist \(c_i\in\R\) and \(q_{ij}\in\Q\) such that
\[
g_i=c_i+\sum_{j=1}^{s}q_{ij}v_j.
\]
For each \(i\), choose a positive integer \(r_i\) such that \(r_iq_{ij}\in\mathbb Z\) for every \(j\). Exponentiating gives
\begin{align}
    \bigl(\exp(g_i)\bigr)^{r_i}
    =\exp(r_i c_i)\prod_{j=1}^{s}\bigl(\exp(v_j)\bigr)^{r_i q_{ij}}.
    \label{eq:algebraic}
\end{align}
The right-hand side belongs to \(\R(\exp(v_1),\ldots,\exp(v_s))\), so \(\exp(g_i)\) is algebraic over this field and hence over \(L\).

Finally, \eqref{eq:subfield} shows that \(L\) is a subfield of \(\R(\exp(g_1),\ldots,\exp(g_n))\),
and \eqref{eq:algebraic} shows that \(\R(\exp(g_1),\ldots,\exp(g_n))\) is algebraic over \(\R(\exp(v_1),\ldots,\exp(v_s))\).
Therefore
\begin{equation}
\label{eq:upper-trdeg}
\trdeg_{\R}L \le \trdeg_{\R} \R(\exp(g_1),\ldots,\exp(g_n))  = \trdeg_{\R}\R(\exp(v_1),\ldots,\exp(v_s)) \leq s.
\end{equation}

\paragraph{Lower bound.}

By definition, \(v_1,\ldots,v_s\) are rationally independent modulo the constant field \(\R\), and for each \(j=1,\ldots,s\),
\[
\frac{\partial\bigl(\exp(v_j)\bigr)}{\exp(v_j)}=\partial v_j.
\]
Moreover, the derivative row matrix \((\partial v_1,\ldots,\partial v_s)\) has rank one over \(\mathcal M(I)\).
If it had rank zero, every \(v_j\) would lie in the constant field \(\R\), contradicting rational independence of \(\{v_1,\ldots,v_s\}\) modulo \(\R\).

Therefore, applying Theorem~\ref{thm:ax-general} with
\(\alpha_j=v_j\) and \(\beta_j=\exp(v_j)\) gives the lower bound
\begin{equation}
\label{eq:lower-trdeg}
\trdeg_{\R}L\geq s+\operatorname{rank}\bigl(\partial v_1,\ldots,\partial v_s\bigr)=s+1.
\end{equation}

\paragraph{Contradiction.}
The bounds \eqref{eq:upper-trdeg} and \eqref{eq:lower-trdeg} contradict each other.
Thus \(|\Crit(F)|\) is finite.
\end{proof}

\section{Application: NPMLEs are finitely supported}
\label{sec:npmle}

Nonparametric maximum likelihood estimation of mixing distributions goes back to Kiefer and Wolfowitz, who established consistency under general regularity conditions without restricting the mixing distribution to a parametric family \cite{kieferWolfowitz1956}.
Lindsay subsequently developed a convex-geometric theory of mixture likelihoods, proving existence under broad conditions and showing that a maximizer can be chosen to have no more support points than there are distinct observations \cite{lindsay1983}.
Lindsay and Roeder further used total positivity to give sufficient conditions for uniqueness of the NPMLE and to relate estimator uniqueness to mixture identifiability \cite{lindsayRoeder1993}.
Subsequent monographs developed the computational and asymptotic theory of nonparametric likelihood estimators in inverse and shape-constrained problems, including interval censoring, deconvolution, density estimation, regression, and mixture models \cite{groeneboom2012information,groeneboom2014nonparametric}.
Here we apply Theorem~\ref{thm:critical-finite} to the NPMLE for the multivariate isotropic Gaussian location-mixture model.

The Gaussian location-mixture model assumes that the observations \(x_1,\ldots,x_n\) are i.i.d.\ draws from the distribution of
\[
X = \Theta + Z,
\]
where \(\Theta\) is a latent location drawn from an unknown mixing distribution \(\pi^*\in\mathcal P(\R^d)\), and \(Z\sim\mathcal N(0,I_d)\) is independent Gaussian noise.

Given a finite dataset \(x_1,\ldots,x_n\), define the average log-likelihood under a candidate mixing distribution \(\pi\) by
\[
\ell_n(\pi)\coloneqq\frac1n\sum_{i=1}^{n}\log p_\pi(x_i),
\]
where \(p_\pi(x)=\int_{\R^d}\varphi_d(x-\theta)\,\pi(d\theta)\) is the mixture density. An NPMLE is any maximizer of this average log-likelihood over all mixing distributions:
\begin{align}
    \label{eq:npmle_definition}
    \widehat\pi\in\argmax_{\pi\in\mathcal P(\R^d)}\frac1n\sum_{i=1}^{n}\log p_\pi(x_i).
\end{align}



Existing results on the Gaussian location-mixture NPMLE fall into three broad directions. Structurally, for general mixture models---not only Gaussian location mixtures---Lindsay established the existence of a discrete NPMLE with at most \(n\) atoms \cite{lindsay1983,lindsayRoeder1993,lindsay1983geometry}. The univariate Gaussian location-mixture NPMLE (\(d=1\)) has additional structure: Lindsay established its uniqueness \cite{lindsay1983geometry}, while Polyanskiy and Wu proved a stronger self-regularization property. Specifically, under a subgaussian true mixing distribution \(\pi^*\), the NPMLE has \(O(\log n)\) support points with high probability, rather than the deterministic upper bound \(n\), and this order is sharp for certain mixtures \cite{polyanskiyWu2020}. Statistically, Saha and Guntuboyina established finite-sample Hellinger-risk bounds for every NPMLE in the multivariate homoscedastic model. Their results yield near-parametric density-estimation rates and near-optimal empirical-Bayes denoising, up to logarithmic factors, for finite mixtures without prior knowledge of the number of components \cite{sahaGuntuboyina2020}. Soloff, Guntuboyina, and Sen extended these guarantees to multivariate heteroscedastic errors, obtaining average-Hellinger bounds, near-optimal denoising oracle inequalities, and adaptive deconvolution guarantees \cite{soloff2025}. In the univariate normal-means model, Chen and Wu derived sharp unregularized regret--Hellinger inequalities and improved NPMLE regret bounds as an empirical Bayes estimator; the compact-support-constrained NPMLE attains the minimax regret rate up to a \(\log\log n\) factor \cite{chen2026sharpregrethellingerboundsgaussian}. Computationally, Yan, Wang, and Rigollet proposed a Wasserstein--Fisher--Rao gradient flow, implemented through particles with alternating location and weight updates, and established convergence guarantees for the associated measure-valued dynamics \cite{yanWangRigollet2024}. For the univariate model, Polyanskiy and Sellke developed certified Wasserstein approximation and finite-time support-size certification, together with almost-sure local linear convergence of EM under suitable generic bounded-data conditions \cite{polyanskiySellke2025}.

In brief, previous work has shown that the NPMLE enjoys near-optimal statistical properties in general dimensions, even with heteroscedastic covariance matrices. In the univariate setting, it also has strong structural properties: the NPMLE is not only unique but often sparse \cite{lindsay1983geometry,polyanskiyWu2020}.
In higher dimensions, however, the NPMLE need not be unique \cite{soloff2025}, and even when the data lies in a bounded domain, the preceding structural sparsity guarantee no longer holds \cite{polyanskiySellke2025}. Previous work has not ruled out an NPMLE with a continuum of support points when \(d\geq2\).

Using Theorem~\ref{thm:critical-finite}, we prove that every finite-sample NPMLE \(\widehat\pi\) is finitely supported.

\maintheoremnpmle*

In the remainder of this section, we sketch the proof of Theorem~\ref{thm:npmle-finite}, prove two standard propositions, and then apply Theorem~\ref{thm:critical-finite} to the Gaussian location-mixture NPMLE.

\paragraph{Sketch of the proof.} Proposition~\ref{prop:likelihood-vector} shows that, for a fixed dataset, all NPMLEs \(\widehat\pi\) have the same fitted likelihood vector. This vector determines a finite Gaussian mixture \(D\), called the dual certificate, whose global maximizers contain the support of every NPMLE. Finally, Theorem~\ref{thm:critical-finite} implies that the set of global maximizers of \(D\) is finite, proving Theorem~\ref{thm:npmle-finite}.

\subsection{Known results for NPMLEs}
In what follows, we fix a dataset \(\{x_1,\ldots,x_n\}\). Propositions~\ref{prop:likelihood-vector} and \ref{prop:dual} are standard results that appear in \cite{polyanskiyWu2020,lindsay1983,lindsayRoeder1993}.
\begin{proposition}[Uniqueness of the fitted likelihood vector]
\label{prop:likelihood-vector}
All NPMLEs \(\widehat\pi\) have the same likelihood vector at the observations,
\[
\widehat{\mathbf p}=(\widehat p_1,\ldots,\widehat p_n),
\qquad
\widehat p_i=p_{\widehat\pi}(x_i).
\]
\end{proposition}

\begin{proof}
For a probability measure \(\pi\), write
\[
\mathbf p(\pi)=\bigl(p_\pi(x_1),\ldots,p_\pi(x_n)\bigr)\in(0,\infty)^n.
\]
The feasible set of likelihood vectors is convex because \(\pi\mapsto\mathbf p(\pi)\) is linear.
The function
\[
(p_1,\ldots,p_n)\longmapsto \frac1n\sum_{i=1}^{n}\log p_i
\]
is strictly concave on \((0,\infty)^n\). Strict concavity therefore implies that the maximizer over the convex set \(\{\mathbf p(\pi):\pi\in\mathcal P(\R^d)\}\) is unique.
\end{proof}

The unique fitted likelihood vector determines a canonical first-order certificate that simultaneously controls the support of every optimizer.

\begin{proposition}[Dual certificate and support containment]
\label{prop:dual}
Define the dual certificate
\begin{equation}
\label{eq:dual}
D(\theta)
=\frac1n\sum_{i=1}^{n}\frac{\varphi_d(\theta-x_i)}{\widehat p_i}.
\end{equation}
Then \(D(\theta)\leq1\) for every \(\theta\in\R^d\).
Moreover, the support of every NPMLE \(\widehat\pi\) is contained in the set of global maximizers of \(D\):
\begin{equation}
\label{eq:candidate-set}
E=\{\theta\in\R^d:D(\theta)=1\}
=\argmax_{\theta\in\R^d}D(\theta).
\end{equation}

\end{proposition}

\begin{proof}
Fix an NPMLE \(\widehat\pi\) and \(\theta\in\R^d\).
For \(0\leq\varepsilon\leq1\), consider the feasible perturbation
\[
\pi_\varepsilon=(1-\varepsilon)\widehat\pi+\varepsilon\delta_\theta.
\]
Taking the right derivative at \(\varepsilon=0\) gives
\[
\left.\frac{d}{d\varepsilon}\ell_n(\pi_\varepsilon)\right|_{\varepsilon=0+}
=\frac1n\sum_{i=1}^{n}\frac{\varphi_d(\theta-x_i)-\widehat p_i}{\widehat p_i}
=D(\theta)-1.
\]
Optimality implies that \(D(\theta)\leq1\) for every \(\theta\in\R^d\).

On the other hand,
\[
\int_{\R^d}D(\theta)\,\widehat\pi(d\theta)
=\frac1n\sum_{i=1}^{n}
\frac{\int\varphi_d(\theta-x_i)\,\widehat\pi(d\theta)}{\widehat p_i}
=1.
\]
Hence \(D=1\) \(\widehat\pi\)-almost surely.
Since \(D\) is continuous, \(\supp(\widehat\pi)\subset E\).
\end{proof}

\subsection{Proof of Theorem~\ref{thm:npmle-finite}}

We can now prove Theorem~\ref{thm:npmle-finite}.

\begin{proof}[Proof of Theorem~\ref{thm:npmle-finite}]
Given the dataset \(\{x_i\}_{i=1}^n\), Proposition~\ref{prop:likelihood-vector} provides a unique fitted likelihood vector \(\widehat{\mathbf p}\) and hence the canonical dual certificate \(D\) in \eqref{eq:dual}.
The certificate \(D\) is a finite isotropic Gaussian location mixture with locations \(x_i\) and strictly positive weights \(w_i=1/(n\widehat p_i)\).
By Theorem~\ref{thm:critical-finite}, its critical set \(\Crit(D)\) is finite.
Every global maximizer of the differentiable function \(D\) is critical, so the canonical set \(E\) in \eqref{eq:candidate-set} is finite.
Proposition~\ref{prop:dual} places the support of every NPMLE \(\widehat\pi\) in this same finite set.
\end{proof}

\section{Discussion}
\label{sec:discussion}

Theorem~\ref{thm:critical-finite} rules out infinite critical sets, and therefore in particular positive-dimensional ones, for every finite isotropic Gaussian location mixture considered in this paper. It also supplies the finiteness hypothesis under which the general mode-counting estimates of \cite{amendola2020,nguyen2026} apply, making their conditional mode bounds unconditional within this homoscedastic class.

\paragraph{Quantitative mode counting.}
The functional-transcendence argument used in this paper is qualitative and does not itself provide an effective estimate for \(\lvert\Crit(F)\rvert\) or for the number of modes. Moreover, the available upper bounds on the numbers of modes and critical points in \cite{amendola2020,nguyen2026} are currently far from sharp. For a three-component homoscedastic mixture, the general bound in \cite{nguyen2026}, once finiteness is established, gives at most \(72\) modes, whereas the specialized analysis in \cite{okunoKabata2026} gives at most \(15\) critical points and at most \(8\) modes.

More specifically, consider the following quantitative mode-counting question.
\begin{question}
    For each \((n,d)\), what are the sharp upper bounds on the numbers of critical points, modes, and global modes of an \(n\)-component isotropic Gaussian mixture in dimension \(d\)?
\end{question}
To the best of the author's knowledge, no conjecture currently addresses these sharp bounds.

For heteroscedastic Gaussian mixtures, Conjecture~5 in \cite{amendola2020}, originating at the 2011 AIM Workshop on Singular Learning Theory, predicted that the largest possible number of modes of a \(d\)-dimensional \(n\)-component mixture is
\[
\binom{d+n-1}{d}. 
\]
But this conjecture was recently disproved by a heteroscedastic three-component mixture in dimension two that violates the proposed bound \cite{kabata2026sevenmodesheteroscedasticthreecomponent}.
In summary, to the best of the author's knowledge, no conjecture currently predicts the sharp mode counting for general \((n,d)\) in either the homoscedastic or heteroscedastic setting.

\paragraph{Possible theoretical extensions.}
A natural analytic extension is to establish an analogue of Theorem~\ref{thm:critical-finite} for finite heteroscedastic Gaussian mixtures with component-specific positive-definite covariance matrices. An affirmative result would also extend the finite-support conclusion to heteroscedastic Gaussian NPMLEs through their corresponding dual certificates \cite{soloff2025}.

\paragraph{Possible algorithmic extensions for NPMLE computation.}
These results suggest a concrete route toward a principled algorithm for computing multivariate NPMLEs. Our finite-criticality theorem guarantees that the dual certificate has only finitely many critical points, so its global maximizers form a finite set of candidate atom locations. The Wasserstein--Fisher--Rao method of \cite{yanWangRigollet2024} provides a practical interacting-particle scheme that updates both particle locations and weights, while the one-dimensional framework of \cite{polyanskiySellke2025} shows how numerical approximations can be accompanied by certificates for Wasserstein error and the exact number of atoms. Combining these ingredients could yield a multivariate algorithm that uses particle dynamics to locate candidate atoms and dual-certificate tests to add, remove, and certify them, with the goal of achieving principled, efficient, and accurate NPMLE computation. Establishing such computational and convergence guarantees remains an important direction for future work.

\paragraph{Open problems for NPMLE.}

Theorem~\ref{thm:npmle-finite} establishes that every finite-sample isotropic Gaussian location-mixture NPMLE is finitely supported, but it neither implies uniqueness nor, by itself, provides a useful sparsity bound. Uniqueness and sparsity are two principal structural features of the univariate NPMLE.
Adversarial datasets have been constructed in \(d\geq2\) for which the NPMLE can be nonunique and can have large support \cite{soloff2025,polyanskiySellke2025}.
However, these constructions use deterministic adversarial datasets rather than typical random samples from a Gaussian mixture model.
It is therefore natural to ask whether, when the data are drawn from a Gaussian mixture, the multivariate NPMLE enjoys the same structural properties as its univariate counterpart \cite{polyanskiyWu2020}. Is the NPMLE almost surely unique? Is it sparse with high probability?

\section*{Acknowledgments}

The author is grateful to his Ph.D. advisor, Yihong Wu, for suggesting the problem studied in this paper, for many useful discussions, and for feedback on the manuscript.

\section*{Declaration of AI use}

AI assistance was used in developing and drafting the proof of Theorem~\ref{thm:critical-finite},
as well as in editing the exposition of this paper.
In particular, the main proof idea for Theorem~\ref{thm:critical-finite}, which combines analytic curve selection with Ax's functional-transcendence theorem, was discovered through the author's interactions with GPT-5.6 Sol Pro via OpenAI's web interface.
All mathematical arguments, statements, and final text were reviewed, verified, simplified, edited, and approved by the author, who takes full responsibility for the contents of the paper.

\bibliographystyle{plainnat}
\bibliography{main_arxiv}

\end{document}